\documentclass[letterpaper, 10 pt, conference]{ieeeconf}

\IEEEoverridecommandlockouts
\usepackage{graphicx}
\usepackage{amsmath} 
\usepackage{amssymb}
\usepackage{cite}

\usepackage{amsthm}
\theoremstyle{definition}
\newtheorem{theorem}{Theorem}

\usepackage{xcolor}
\title{\LARGE \bf
Nonlinear System Level Synthesis for Polynomial Dynamical Systems
}

\author{Lauren Conger, Jing Shuang (Lisa) Li, Eric Mazumdar, Steven L. Brunton%
\thanks{Conger, Li, and Mazumdar are with the Division of Engineering and Applied Sciences at the California Institute of Technology and Brunton is with the Department of Mechanical Engineering at the University of Washington.
{\tt\small \{lconger, jsli, mazumdar\}@caltech.edu}, {\tt\small sbrunton@uw.edu}}%
}

\begin{document}

\maketitle
\thispagestyle{empty}
\pagestyle{empty}

\begin{abstract}
This work introduces a controller synthesis method via system level synthesis for nonlinear systems characterized by polynomial dynamics. 
The resulting framework yields finite impulse response, time-invariant, closed-loop transfer functions with guaranteed disturbance cancellation. Our method generalizes feedback linearization to enable partial feedback linearization, where the cancellation of the nonlinearity is spread across a finite-time horizon. 
This provides flexibility to use the system dynamics to attenuate disturbances before cancellation via control, reducing the cost of control compared with feedback linearization while maintaining guarantees about disturbance rejection. This approach is illustrated on a benchmark example and on a common model for fluid flow control.
\end{abstract}

\section{INTRODUCTION}
There is a large body of work on the control of nonlinear systems~\cite{Khalil,charlet1989dynamic}, with several widely deployed methods, such as model predictive control (MPC)~\cite{Korda2018automatica,Kaiser2018prsa,peitz2019koopman,folkestad2021koopman}, reinforcement learning~\cite{Sutton1998book}, feedback linearization~\cite{charlet1989dynamic,Khalil}, and control Lyapunov functions~\cite{Ames1,Ames2}, among many others.
However, general-purpose approaches to designing stabilizing controllers for large classes of nonlinear dynamical systems remains an open problem. 
In this work, we take steps towards this goal by leveraging new approaches stemming from system level synthesis (SLS) theory to develop stabilizing controllers for the general class of polynomial dynamical systems. This results in a controller that provides guaranteed stabilization in finite time with respect to an initial condition and disturbances. SLS provides a framework of guarantees for disturbance rejection and stabilization in terms of transfer functions that map from the disturbance to the state and input~\cite{Anderson}. Our control method is a generalization of feedback linearization, where instead of using feedback to completely linearize a nonlinear system, we partially linearize the system, optimizing over the amount of feedback linearization at each time step before full cancellation. The optimization of this highly non-convex problem is difficult to solve, but our convergence guarantees hold for a range of parameters, removing the risks of learning sub-optimal solutions. The choice of cost function is interchangeable and can be tailored to specific application areas. 

Previous works in linear and robust SLS ~\cite{Anderson} provide a parameterization that allows for system level constraints and localization of disturbances. Advances in nonlinear SLS provide criteria that must be satisfied by closed-loop transfer functions~\cite{Ho} and methods for controlling nonlinearities induced by saturation of input and state constraints~\cite{Yu}. Current SLS techniques for nonlinear systems include MPC with local linearization~\cite{AmoAlonso1,AmoAlonso2}. We expand on these works by explicitly defining nonlinear transfer functions for polynomial systems, removing the approximation of local linearization.
This class of systems includes a wide range of phenomena, such as the dynamics of fluid flows and other convectively nonlinear systems~\cite{Noack2003jfm}. 
Other systems may often be approximated by polynomial dynamics, either locally or globally, for example through Taylor series expansion~\cite{Brunton} or normal form expansion~\cite{guckenheimer_holmes}. 
 Moreover, recent work by Qian et al.~\cite{Qian} presents a method for learning a lifting map to transform general nonlinear dynamics into a quadratic form. Another recent work~\cite{Borggaard} details approximating feedback controllers for quadratic dynamics with quadratic costs for continuous systems; these quadratic systems fall under the class of systems in this paper.

\section{Preliminaries and Problem Statement}

This work constructs stabilizing controllers for discrete-time nonlinear dynamical systems through the use of SLS. We consider time-invariant nonlinear dynamics
\begin{equation}\label{eq:nonlinear_dynamics}
    x_{t+1} = f(x_t,u_t) + w_t,
\end{equation}
where $x_t \in \mathbb{R}^n$ is the state, $u_t \in \mathbb{R}^m$ is the input, and  $w_t \in \mathbb{R}^n$ is the disturbance at time $t$. For simplicity and without loss of generality, we set $x_0=0$ and drive the system to $x=0$.

We approach controller synthesis by defining closed-loop maps from the disturbances and initial condition to the state and input. We then use SLS to guarantee that the controller is stabilizing. 
The remainder of this section defines notation and reviews background on SLS and nonlinear control. 

\subsection{Notation}
Let $x_t$ denote the system state at time $t$, and $u_t$, $w_t$ denote the control input and disturbance, respectively. Let $x^{(i)}$ be the $i^{th}$ vector index of $x$, and let $x^i$ (without brackets on $i$) be $x$ to the power $i$. Let $w_{t:t-k}$ represent the list of disturbances $(w_t,w_{t-1},\dots,w_{t-k})$. Capital letters represent functions or matrices mapping vectors to vectors, and lower case letters denote scalars. The symbol $\otimes$ denotes the column-wise Kronecker product, defined for $x=[x^{(1)},x^{(2)},\dots,x^{(n)}]^\top$ and $y=[y^{(1)},y^{(2)},\dots,y^{(n)}]^\top$, as 
\begin{align*}
    x \otimes y = & \big[x^{(1)} y^{(1)}, x^{(1)} y^{(2)}, \dots, \\ &x^{(1)} y^{(n)}, x^{(2)} y^{(1)}, x^{(2)} y^{(2)}, \dots\ \dots, x^{(n)} y^{(n)}  \big]^\top
\end{align*}
and $x^{\otimes j}$ denotes $j-1$ Kronecker products of $x$ with itself; e.g., 
    $x^{\otimes 4} = x \otimes x \otimes x \otimes x$.

\subsection{System Level Synthesis}

The key idea of the SLS framework is that instead of formulating a control problem as an optimization over a space of controllers, we formulate the control problem as an optimization over a space of \textit{closed-loop maps} (CLMs). CLMs map exogenous signals that the engineer has no control over (e.g. process disturbance, sensor noise) onto signals of interest that can be affected (e.g. state and control input). This parametrization is especially beneficial for learning-plus-control approaches, as it establishes a relationship between model uncertainty and closed-loop system performance~\cite{Anderson, Dean2019}. In this work, we will restrict ourselves to the state feedback case. We define CLMs from disturbance $w$ to the state $x$ and control input $u$ as $\Psi^x$ and $\Psi^u$, respectively:
\begin{equation}\label{eq:clms}
\begin{split}
    x_t &= \Psi_t^x(w_{t:0}) \\
    u_t &= \Psi_t^u(w_{t:0}).
\end{split}
\end{equation}
Key to our derivation is the following result on CLMs for nonlinear systems.
\begin{theorem} \label{thm:sls_theorem} (Theorem III.3 in \cite{Ho})
For system~\eqref{eq:nonlinear_dynamics}, CLMs $\Psi^x$, $\Psi^u$ are achievable by a causal state feedback controller $u_t = K_t(x_t)$ if and only if they satisfy the following: 
\begin{equation} \label{eq:SLS_condition}
    \Psi^x_t(w_{t:0}) = f(\Psi_{t-1}^x(w_{t-1:0}),\Psi_{t-1}^u(w_{t-1:0})) + w_t.
\end{equation}
We refer to \eqref{eq:SLS_condition} as the \textit{SLS achievability constraint}. Given any CLMs $\Psi^x$, $\Psi^u$ satisfying~\eqref{eq:SLS_condition}, we are able to construct a causal state feedback controller (see \cite{Ho} for details) that achieves these CLMs.
\end{theorem}
Theorem \ref{thm:sls_theorem} ensures that no viable causal controllers are lost by switching to a parametrization over CLMs; furthermore, the theorem guarantees that one can always reconstruct a causal controller that matches the desired CLMs. We note that~\eqref{eq:SLS_condition} is simply a rewriting of~\eqref{eq:nonlinear_dynamics} so that $u_t$ and $x_t$ are functions of the disturbances rather than previous $x_t$ and $u_t$.

Our analysis is limited to time-invariant CLMs, since $f$ is assumed to be time-invariant; the time-varying case is future work. Therefore, we omit the $t$ subscript on $\Psi^x$ and $\Psi^u$. 

Of particular interest are controllers that achieve a finite impulse response (FIR). CLMs have FIR with time horizon $T$ if they can be written as:
\begin{align}
\label{eq:SLS_condition_finite}
\begin{split}
    \Psi^x(w_{t:0}) &= \Psi^x(w_{t:t-T}) \\
    \Psi^u(w_{t:0}) &= \Psi^u(w_{t:t-T})
\end{split}
\end{align}
where instead of depending on the entire history of disturbance $w$ (i.e. $w_{t:0}$), CLMs only depend on the most recent $T$ values of $w$ (i.e. $w_{t:t-T}$). This can be interpreted as rejecting disturbances after $T$ timesteps. In the subsequent sections we will describe how to construct FIR CLMs in the form of~\eqref{eq:SLS_condition_finite} for polynomial dynamical systems. 

\subsection{Nonlinear Control}
Nonlinear control has a rich and extensive history \cite{kowalski1991nonlinear,Krstic,charlet1989dynamic,Isidori,Khalil} and is an active area of research. 
Dominant techniques include feedback linearization and control Lyapunov functions (CLFs). Feedback linearization cancels nonlinear terms using the input, and then stabilizes resulting linear system dynamics using traditional linear methods~\cite{charlet1989dynamic,Ames1,Isidori,Khalil,Krstic,Sastry}. However, depending on the current state, the nonlinearities can drive the system toward its desired state, and this cancellation unnecessarily uses large inputs~\cite{Ames1,Ames2}. CLFs offer a method for controller design where the control input is based on a certificate for exponential stability; however no formulaic method for finding a certificate exists and a clever selection of the certificate is necessary to avoid overly conservative control. We use this for comparison with our examples; our FIR CLMs remove the need for an exponential stability certificate and also reduce unnecessarily large inputs introduced by feedback linearization. We structure the problem to optimize over the amount of cancellation over a fixed time window by parameterizing the feedback linearization magnitude of each nonlinear term.
\section{Finite Impulse Response for Closed-Loop Maps}
We construct a class of controllers for scalar polynomial systems, followed by vector polynomial systems. The controllers are parameterized by coefficients $\alpha_j^{(k)} \in [0,1]$, over which the controllers are optimized. We show that CLMs under the controllers are stabilizing with respect to disturbances, and comment on the connection between the $\alpha_j^{(k)}$ parameters and feedback linearization, particularly how $\alpha_j^{(0)}=1$ corresponds to immediate feedback linearization. The relationships among the disturbances, transfer functions, and coefficients are illustrated in figure \ref{fig:diagram}.

\begin{figure*}[thpb]
  \centering
  \includegraphics[width=.8\textwidth,trim=110 65 110 120, clip]{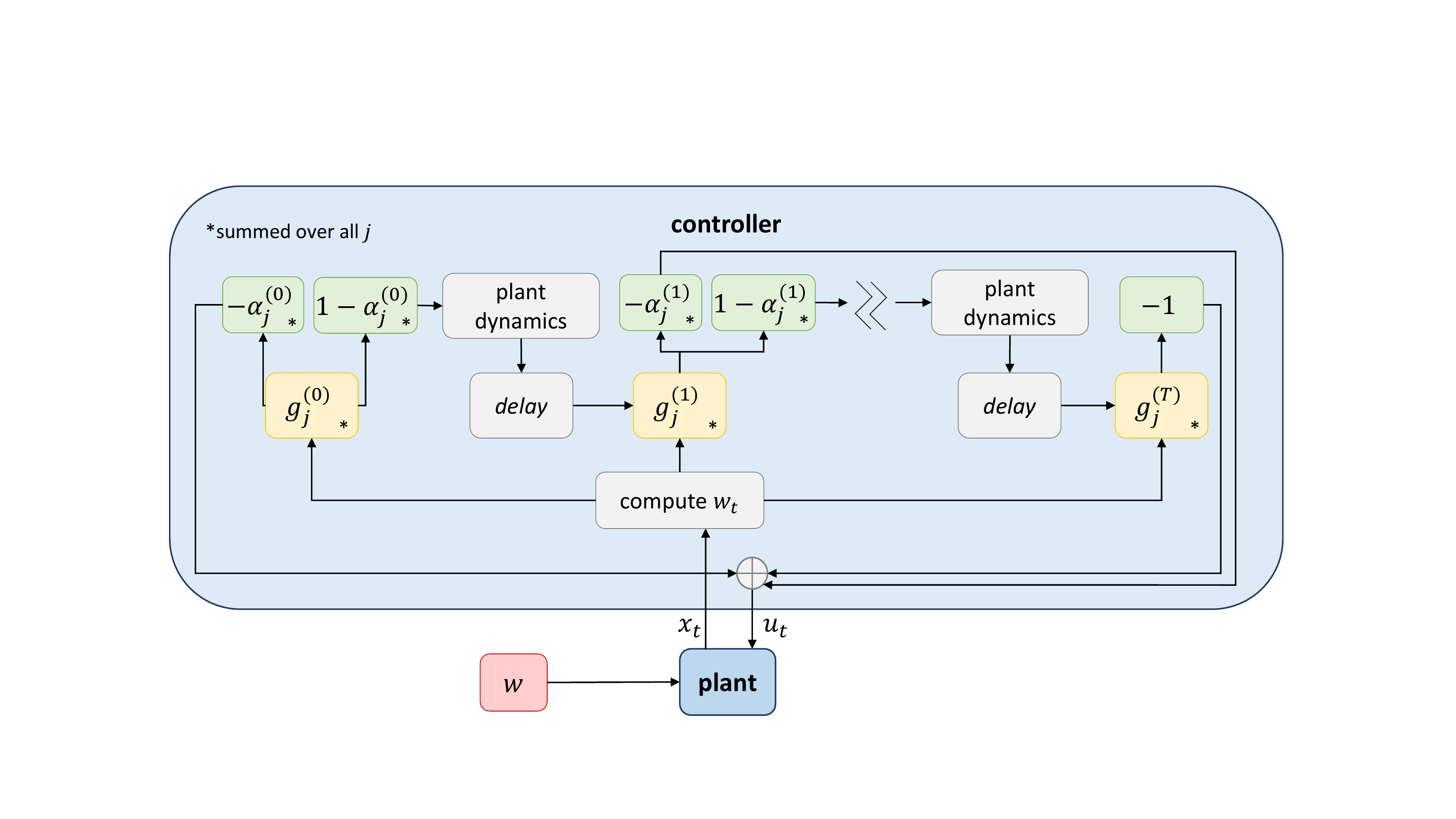}
  \vspace{-.1in}
  \caption{The control input and state are functions of the disturbances. The $\alpha_j^{(k)}$ parameters determine how the effect of the disturbance is distributed between the state and controller transfer functions. }
  \label{fig:diagram}
\end{figure*}

\subsection{Scalar Polynomial Case}
First we prove results for the scalar case for pedagogical ease. Consider the nonlinear difference equation for $x_t \in \mathbb{R}$
\begin{equation}\label{eq:scalar_polynomial}
    x_{t+1} = \sum_{j=1}^n a_j x_t^j+ u_t + w_t,
\end{equation}
where $u_t, w_t \in \mathbb{R}$ and $n$ is a finite positive integer. We consider FIR CLMs $\Psi^x(w_{t:t-T})$ and $\Psi^u(w_{t:t-T})$.
We define the set of functions $ \{g_j^{(k)}(w_{t:t-k})\}_{j=1}^{c_k}$, parameterized by $\alpha_j^{(k)} \in [0,1]$, as the set of monomial functions of $w_{t:t-k}$. Here the superscript $(k)$ indicates the length of the disturbance time window and the subscript $j$ indicates the index of the monomial. The value of $c_1$ depends on the value of $n$; it is the number of monomials that do not depend on $w_t$ after completely expanding the left hand side of \eqref{eq:scalar_zero_to_one}. We define the multi-index $p^{(j,k)}$, where $p^{(j,k)}=(p^{(j,k)}_0,p_1^{(j,k)},\dots,p_{k}^{(j,k)})$, to denote the powers of each disturbance in monomial terms. Let $$w_{t:t-k}^{p^{(j,k)}} = w_t^{p_0^{(j,k)}} w_{t-1}^{p_1^{(j,k)}} \dots w_{t-k}^{p_k^{(j,k)}};$$ the coefficients $b_j^{(k)}$ are constants. Then $g_j^{(k)}$ are defined as 
\begin{equation}\label{eq:g}
    g_j^{(k)}(w_{t:t-k}) = b_j^{(k)} w_{t:t-k}^{p^{(j,k)}}.
\end{equation}

The functions $g_j^{(k)}(w_{t:t-k})$ are computed recursively, for all $j$ from $1$ to $c_0=n$:

\begin{equation*}
    \begin{split}
        g_j^{(0)}(w_t) &= a_j w_t^j \\
        g_j^{(1)}(w_{t:t-1}) &= b_j^{(1)} w_{t:t-1}^{p^{(j,1)} }\\
            \end{split}
\end{equation*}
where $p^{(j,k)}$ and $b_j^{(k)}$ are such that $g_j^{(1)}$ satisfies

\begin{equation*}
    \begin{split}
        &\sum_{j=1}^{c_1} g_j^{(1)}(w_{t:t-1}) + \sum_{j=1}^{n} a_k w_{t}^k \\
        &\qquad \qquad =\sum_{j=1}^n a_j \left( \sum_{k=1}^n (1-\alpha_k^{(0)}) a_k w_{t-1}^k + w_{t} \right)^j.
        \end{split}\end{equation*}
Equivalently, we obtain $g_j^{(1)}$ from $g_j^{(0)}$ as 
\begin{equation}
    \begin{split}\label{eq:scalar_zero_to_one}
        &\sum_{j=1}^{c_1} g_j^{(1)}(w_{t:t-1}) + \sum_{j=1}^{n} g_j^{(0)}(w_{t}) \\
        &\qquad \qquad =\sum_{j=1}^n a_j \left( \sum_{k=1}^n (1-\alpha_k^{(0)}) g_k^{(0)}(w_{t-1}) + w_t \right)^j.
    \end{split}
\end{equation}
Monomial terms depending on only $w_t$ naturally consist of the $g_j^{(0)}(w_t)$ terms. Similar to the zero-to-one case, the value of $c_m$ depends on $n$ and is defined as the number of monomials that depend strictly on $w_t$ through $w_{t-m-1}$. We now obtain $g^{(k+1)}$ from $g^{(k)}$: 
\begin{equation}\label{eq:scalar_recursion}
\begin{split}
    &\sum_{m=0}^T \sum_{j=1}^{c_m} g_j^{(m)}(w_{t:t-m}) \\ 
     &=\sum_{l=1}^n a_l \left(\sum_{m=1}^{T} \sum_{j=1}^{c_m}(1-\alpha_j^{(m)})g_j^{(m)}(w_{t-1:t-m}) + w_t\right)^l.
\end{split}
\end{equation}

 Equipped with this notation, we present the closed loop transfer functions in Theorem \ref{thm:scalar_polynomial}. Note that $\alpha_j^{(k)}$ are the design parameters, while $b_j^{(k)}$ and $p^{(j,k)}$ are inherent to the system dynamics.
 
\begin{theorem}\label{thm:scalar_polynomial}
The closed loop transfer functions 
    \begin{align}
        \Psi^x(w_{t:t-T}) =& \sum_{k=0}^{T-1} \sum_{j=1}^{c_k}(1-\alpha_j^{(k)})g_j^{(k)}(w_{t-1:t-1-k}) + w_t \nonumber\\
        \Psi^u(w_{t:t-T}) =& - \sum_{k=0}^{T-1} \sum_{j=1}^{c_k} \alpha_j^{(k)} g_j^{(k)}(w_{t:t-k})\nonumber \\
        & -\sum_{j=1}^{c_T} g_j^{(T)}(w_{t:t-T})  \label{eq:clm_scalar}
    \end{align}
satisfy the SLS achievability constraint \eqref{eq:SLS_condition} for the nonlinear dynamics defined by \eqref{eq:scalar_polynomial}, resulting in a FIR stabilizing controller for all $\alpha_j^{(k)} \in [0,1]$, where $g_j^{(k)}$ are defined by \eqref{eq:g}.
\end{theorem}
\begin{proof}
We consider FIR transfer functions. For a system with dynamics defined by \eqref{eq:scalar_polynomial}, the SLS achievability constraint \eqref{eq:SLS_condition} can be written as:
\begin{equation} \label{eq:SLS_condition_scalar}
    \begin{split}
        \Psi^x(w_{t:t-T}) = \sum_{j=1}^n a_j \left(\Psi^x(w_{t-1:t-1-T})\right)^j \\
        + \Psi^u(w_{t-1:t-1-T}) + w_t.
    \end{split}
\end{equation}
Plugging \eqref{eq:clm_scalar} into \eqref{eq:SLS_condition_scalar} gives the following equality, which we wish to satisfy:
\begin{align*}
         &\sum_{k=0}^{T-1} \sum_{j=1}^{c_k}(1-\alpha_j^{(k)})g_j^{(k)}(w_{t-1:t-1-k})+ w_t\\ 
         &= \sum_{j=1}^n a_j \left(  \sum_{k=0}^{T-1} \sum_{j=1}^{c_k}(1-\alpha_j^{(k)})g_j^{(k)}(w_{t-2:t-2-k}) + w_{t-1}\right)^j\\
         &\qquad - \sum_{k=0}^{T-1} \sum_{j=1}^{c_k} \alpha_j^{(k)} g_j^{(k)}(w_{t-1:t-1-k})\\
         &\qquad -\sum_{j=1}^{c_T}g_j^{(T)}(w_{t-1:t-1-T}) + w_t.
\end{align*}
The $w_t$ term appears on both sides, and cancels out. We can rewrite the first term on the right hand side using the recursive definition of $g$, per \eqref{eq:scalar_recursion}. Then \eqref{eq:SLS_condition} is satisfied if the following equality holds:
\begin{align*}
        &\sum_{k=0}^{T-1} \sum_{j=1}^{c_k} \left( 1-\alpha_j^{(k)} \right) g_j^{(k)}(w_{t-1:t-1-k})\\ &\qquad = \quad \sum_{k=0}^{T} \sum_{j=1}^{c_k} g_j^{(k)}(w_{t-1:t-1-k}) \\
         &\qquad \quad-\sum_{j=1}^{c_T}g_j^{(T)}(w_{t-1:t-1-T}) \\
         &\qquad \quad- \sum_{k=0}^{T-1} \sum_{j=1}^{c_k} \alpha_j^{(k)} g_j^{(k)}(w_{t-1:t-1-k}).
\end{align*}
On the right hand side, the $g_j^{(T)}$ terms cancel and the two remaining sums are equal:
\begin{equation*}
\begin{split}
        = &\quad \sum_{k=0}^{T-1} \sum_{j=1}^{c_k} g_j^{(k)}(w_{t-1:t-1-k}) \\
         &\qquad \quad-\sum_{j=1}^{c_T} \left( g_j^{(T)}(w_{t-1:t-1-T}) - g_j^{(T)}(w_{t-1:t-1-T}) \right )\\
         &\qquad \quad-\sum_{k=0}^{T-1} \sum_{j=1}^{c_k} \alpha_j^{(k)} g_j^{(k)}(w_{t-1:t-1-k}) \\
        = &\quad \sum_{k=0}^{T-1} \sum_{j=1}^{c_k} \left( 1-\alpha_j^{(k)} \right ) g_j^{(k)}(w_{t-1:t-1-k}).
\end{split}
\end{equation*}
By Theorem \ref{thm:sls_theorem}, any continuous close-loop map that satisfies \eqref{eq:SLS_condition} results in a stabilizing controller, which we show above. The controller is given by
\begin{equation}
    u_t = \Psi^u(w_{t:t-T})
\end{equation}
where $\Psi^u(w_{t:t-T})$ defined in \eqref{eq:clm_scalar}. Because we consider the state feedback case, we can compute prior disturbances by subtracting the expected state from the actual state. The disturbance at time $t$ is computed by
\begin{equation*}
    w_t = x_{t+1} - \sum_{j=1}^n a_j x_t^j - u_t.
\end{equation*}
\end{proof}
Note that the term $\sum_{j=1}^{c_T} g_j^{(T)}(w_{t:t-T})$ in $\Psi^u$ is the only term depending on the least recent disturbance $w_{t-T}$. This term cancels out any remaining effects from $w_{t-T}$, and is what makes the transfer functions have a finite-time horizon.

\subsection{Vector Polynomial Case}
Now we generalize Theorem \ref{thm:scalar_polynomial} to the case where $x_t \in \mathbb{R}^n$. The dynamics we consider take the form
\begin{equation}\label{eq:vector_dynamics}
    \begin{split}
        x_{t+1} = \sum_{j=1}^n H_j x_t^{\otimes j} + u_t + w_t
    \end{split}
\end{equation}
where $H_j \in \mathbb{R}^{n^j}$ and $u_t,\ w_t \in \mathbb{R}^n$.
\paragraph*{Remark} Key to the derivation is the assumption that $u_t \in \mathbb{R}^n$, meaning that the controls are potentially unbounded, and that we have direct control over each state. We define a set of functions $G_j^{(k)}(w_{t:t-k})$ that are vectors of monomial functions of entries of the disturbance vectors $w_{t:t-k}$. The functions are parametrized by coefficient matrices $B_j^{(k)}$ and multi-indices $p^{(j,k)}$ as in the scalar case:
\begin{equation*}
    \begin{split}
    G_j^{(k)} &= B_j w_{t:t-k}^{\otimes p^{(j,k)}} \\
    w_{t:t-k} &= w_t^{\otimes p_0^{(j,k)}} \otimes w_{t-1}^{\otimes p_1^{(j,k)}} \otimes \dots \otimes w_{t-k}^{\otimes p_k^{(j,k)}}.
    \end{split}
\end{equation*}

Again, we define the functions recursively starting at $G_j^{(0)}$. The matrices $G_j^{(0)}$ correspond to the scalar $g_j^{(k)}$ in the scalar case.
\begin{equation*}
\begin{split}
        G_j^{(0)}(w_t) &= H_j w_t^{\otimes j} \\
        G_j^{(1)}(w_t) &= B_j w_t^{\otimes p_0^{(j,1)}} \otimes w_{t-1}^{\otimes p_1^{(j,1)}}, \\
    \end{split}
\end{equation*}
where $B_j$ and $p^{(j,1)}$ are such that
\begin{equation*}
\begin{split}
        &\sum_{j=1}^n H_j \left(\sum_{k=1}^n (1-A_k^{(0)}) H_k w_{t-1}^{\otimes k} + w_t \right)^{\otimes j} \\&\qquad  =\sum_{j=1}^{c_1}G_j^{(1)}(w_{t:t-1}) + \sum_{j=1}^n H_j w_t^{\otimes j},
    \end{split}
\end{equation*}
or equivalently in terms of $G_j^{(0)}$,
\begin{equation}
    \begin{split}
        \sum_{j=1}^n H_j \left(\sum_{k=1}^n (1-A_k^{(0)}) G_j^{(0)}(w_{t-1}) + w_t \right)^{\otimes j} \\ =\sum_{j=1}^{c_1}G_j^{(1)}(w_{t:t-1}) + \sum_{j=1}^n G_j^{(0)}(w_t).
    \end{split}
\end{equation}
We denote $c_k$ as the appropriate number of monomial terms. The general recursion equation is given by
\begin{align*}
      \sum_{l=1}^n H_l \left(\sum_{m=1}^{T} \sum_{j=1}^{c_m}(1-A_j^{(m)})G_j^{(m)}(w_{t-1:t-m}) + w_t\right)^{\otimes l} \\
    = \sum_{m=0}^T \sum_{j=1}^{c_m} G_j^{(m)}(w_{t:t-m}).  
\end{align*}
\begin{proof}
The proof follows identically to the proof for the scalar case. Note that the number of coefficients is significantly greater for the vector case.
\end{proof}

\subsection{Relationship to Feedback Linearization}
We designed the CLMs $\Psi^x$ and $\Psi^u$ to be a generalization of feedback linearization. We discuss two observations regarding $\alpha_j^{(k)}$ (equivalently, $A_j^{(k)}$):
\begin{enumerate}
    \item[I] The value of $\alpha_j^{(k)}$ determines the amount of feedback linearization for monomial $j$.
    \item[II] The value of $\alpha_j^{(k)}$ reflects the assignment of feedback linearization dependent on disturbances up to $k$ time steps prior to current actuation.
\end{enumerate}
\subsubsection{Amount of Feedback Linearization}
If we were using pure feedback linearization, we would select $\alpha_j^{(k)}=1$ corresponding to all $g_j^{(k)}(w_{t:t-k})$ that are nonlinear in $w$, which would result in a control input $u_t$ that completely cancels the nonlinear terms.  The choice of $\alpha_j^{(k)}$ taking values between zero and one offers the flexibility to use feedback linearization only when it is necessary to drive the state to zero. CLFs outperform feedback linearization by allowing the dynamics to naturally dampen disturbances; our controller behaves similarly.

\subsubsection{Feedback Linearization over Time}
When we employ pure feedback linearization, that is, $\alpha_j^{(k)}=1$ for all nonlinear terms, the controller $u_t$, parameterized by $\Psi^u$, cancels the disturbance at the first possible time step, rendering the dynamics linear. When we do not immediately cancel the disturbances, they propagate through the dynamics before they are cancelled $T$ iterations after introduction to the system. When such propagation reduces the effect of the disturbance on the state, we use less control effort by cancelling later. 
\paragraph*{Remark} Although we consider $\alpha_j^{(k)}$ to be constants in this paper, they could be functions of $w_{t:t-k}$ which would improve performance. This would be even more similar to CLFs, where the controller parameters depend on the state.
\paragraph*{Remark} The functions $g_j^{(k)}$, $G_j^{(k)}$ contain the complicated nesting of the polynomial recursions, and are in general nontrivial to compute. In our examples we use Mathematica for these computations; generalizing this process is future research.

\section{Optimizing the Controller}
Now that we have a set of functions parameterizing the CLMs for our polynomial dynamics, we seek to choose coefficients $\alpha_j^{(k)}$ or $A_j^{(k)}$ that provide the best performance relative to a cost function. In linear SLS, the cost function is a function of the transfer function matrices $\Psi^x$ and $\Psi^u$, which results in a cost that is independent of the disturbances. Here we present two definitions of cost in terms of these nonlinear operators, depending on the user's knowledge of the disturbances. Our goal is to minimize the cost, $J(\Psi^x,\Psi^u)$, over functions $\Psi^x$ and $\Psi^u$ that take the form given in Theorem \eqref{thm:scalar_polynomial}.

\subsection{Known Disturbance Distribution}
Consider the case where the user has a reasonable estimate of the disturbance distribution $W$ where $w_t \sim W$. We define the cost of the CLMs, for $Q \succeq 0$, $R \succ 0$, as
\begin{align}\label{eq:quadratic_cost_function}
    \begin{split}
        J(\Psi^x,\Psi^u) = \mathbb{E}_W \big[& \Psi^x(w_{t:t-T})^\top Q \Psi^x(w_{t:t-T})  \\ & +\Psi^u(w_{t:t-T})^\top R \Psi^u(w_{t:t-T}) \big].
    \end{split}
\end{align}

This definition is useful for optimizing with respect to the effect of the nonlinearities in the dynamics which could depend, for example, on the sign of the noise. We minimize this in example \ref{ex:scalar_quadratic} using gradient descent libraries on a large number of disturbances sampled from a uniform distribution.
\subsection{Unknown Disturbances}
To present a cost function comparable to traditional methods, we present an induced norm on the transfer function that is useful when the disturbances are unknown and we respond to the worst-case noise sequences. For a polynomial function of order $N$ with an FIR horizon of $T$, the cost is given by

\begin{align*}
    &J(\Psi^x,\Psi^u) &= \max_{\left\|w_{t:t-T}\right\| \leq 1} \left \| C \Psi^x(w_{t:t-T}) + D \Psi^u(w_{t:t-T}) \right \|,
\end{align*}
for any induced vector norm.

\section{Examples}
\subsection{Scalar Quadratic}\label{ex:scalar_quadratic}
Consider the dynamical system given by
\begin{equation}\label{eq:scalar_quadratic}
    x_{t+1} = x_t^2 - x_t +u_t + w_t
\end{equation}
where $x_t,u_t,w_t \in \mathbb{R}$. As shown in the scalar polynomial section, we want to find a time-invariant parameterization for $u_t$ that depends upon a fixed number of previous disturbances, defined to be $T$. We begin by defining the set of functions $g_j^{(k)}$ for this particular system. By definition, the values of $g_j^{(0)}$ are given by
\begin{equation*}
    \begin{split}
        g_1^{(0)}(w_t) &= -w_t \\
        g_2^{(0)}(w_t) &= w_t^2.
    \end{split}
\end{equation*}
Using the recursion equation, we determine that the set of $g_j^{(1)}$ must satisfy
\begin{equation*}
    \begin{split}
        \left( (1-\alpha_2^{(0)}) w_{t-1}^2 - (1-\alpha_1^{(0)}) w_{t-1} + w_t \right)^2 \\+ (1-\alpha_2^{(0)}) w_{t-1}^2 - (1-\alpha_1^{(0)}) w_{t-1} + w_t \\
        = \sum_{j=1}^{c_1}g_j^{(1)}(w_{t:t-1}) - w_t + w_t^2.
    \end{split}
\end{equation*}
Let $\beta_1 = (1-\alpha_1^{(0)})$ and $\beta_2 =(1-\alpha_2^{(0)})$. Then $c_1 = 6$ and the functions are given by
\begin{align*}
    g_1^{(1)} &= \beta_2^2 w_{t-1}^4  &g_2^{(1)} &= 2\beta_1 \beta_2 w_{t-1}^3 \\
    g_3^{(1)} &= (\beta_1^2 -\beta_2) w_{t-1}^2  &g_4^{(1)} &= 2\beta_2 w_{t-1}^2 w_t \\
    g_5^{(1)} &= -2\beta_1 w_{t-1} w_t &g_6^{(1)} &= \beta_1 w_{t-1}.
\end{align*}
Here, $p^{(1,0)}=[1]$, $p^{(2,0)}=[2]$, $p^{(1,1)}=[0,4]$, $p^{(4,1)}=[1,2]$, etc.. For the sake of brevity, we fix $T=2$ and do not list the $g_j^{(2)}$ functions because we have $c_2 = 26$. The controller $u_t$ is given by 
\begin{equation*}
    u_t = -\sum_{j=1}^{26} g_j^{(2)}(w_{t:t-T}) - \sum_{j=1}^6 \alpha_j^{(1)}g_j^{(1)}(w_{t:t-1}) - \sum_{j=1}^2 \alpha_j^{(0)} 
    g_j^{(0)}. 
\end{equation*}
To illustrate the effect of the alpha parameters, we fix all but one $\alpha_j^{(k)}$ to values determined by gradient descent, varying only $\alpha_2^{(0)}$. We compute the total cost $J$ given by \eqref{eq:example_cost_function} over 100 trials of random sequences of noise, each of length 23.  
\begin{align}\label{eq:example_cost_function}
    J(\Psi^x,\Psi^u) = \left( \Psi^x \right)^2 + \left( \Psi^u \right)^2.
\end{align}

Simulation results for this example are shown in figure \ref{fig:scalar_quad_results}. The cost is smooth and convex in $\alpha_2^{(0)}$, and outperforms a hand-tuned CLF with the Lyapunov equation given by $V(x) = x^2$. Note that we are optimizing one term out of 34, so the effect from a single $\alpha$ parameter will be small relative to the total cost.
\begin{figure}[t!]
  \centering
  \includegraphics[width=.235\textwidth,trim=0 15 0 5, clip]{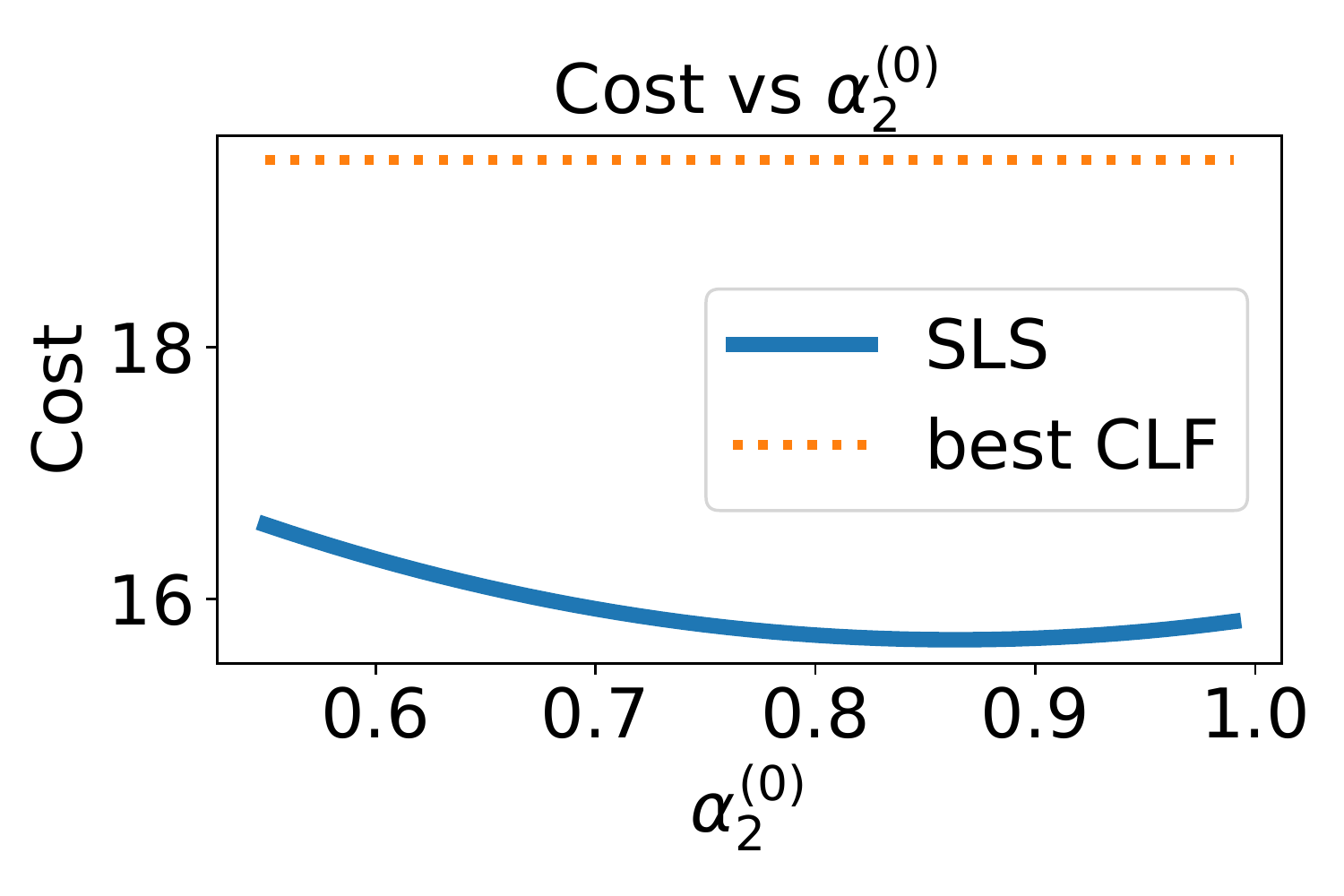}
  \includegraphics[width=.235\textwidth,trim=0 15 0 5, clip]{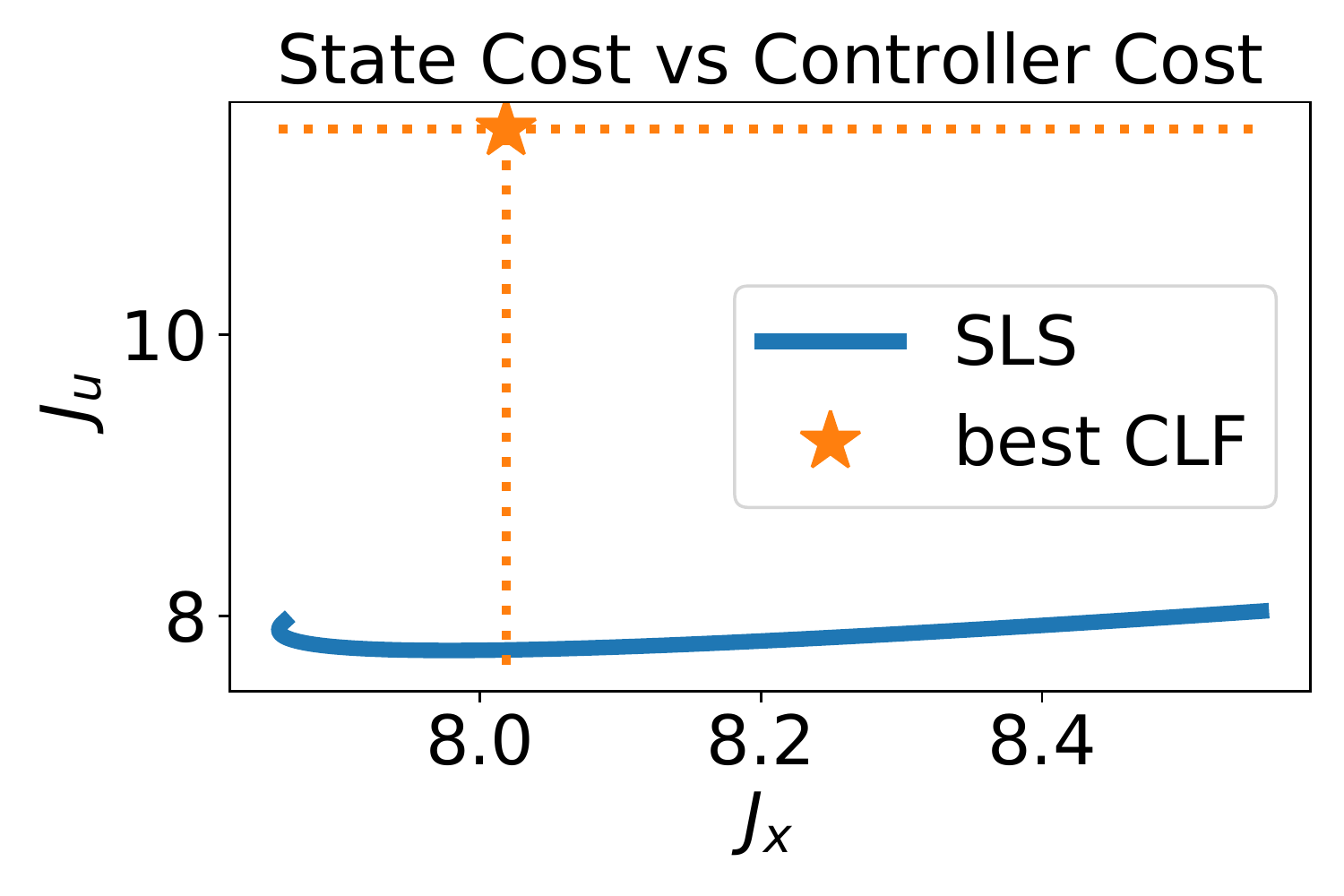}
  \vspace{-.1in}
  \caption{The values chosen for the $\alpha_j^{(k)}$ coefficients determine the assignment of nonlinearity to the state and controller transfer functions. We can optimize the total cost with respect to these coefficients.}
  \label{fig:scalar_quad_results}
\end{figure}

\subsection{Vortex Shedding Behind an Obstacle}
The nonlinear dynamics governing the fluid flow past a circular cylinder can be represented as a system of three quadratic ordinary differential equations ~\cite{Noack2003jfm}. This quadratic system is of interest to researchers in fluid mechanics and related fields. The dynamics can be written as a nonlinear difference equation given by 
\begin{equation}
\begin{split}
    x_{t+1} &= \mu x_t - \omega y_t + a x_t y_t + w_t^v\\
    y_{t+1} &= \omega x_t + \mu y_t + a y_t z_t + u_t + w_t^y \\
    z_{t+1} &= -\lambda (z_t-x_t^2-y_t^2) + w_t^z.
\end{split}
\end{equation}
Note that this takes the same form as \eqref{eq:vector_dynamics} where $\mathbf{x_t} = [x_t,\ y_t,\ z_t]^\top$, $\mathbf{w_t}=[w_t^v,\ w_t^y,\ w_t^z]$, and $\mathbf{u_t} = [0,\ u_t,\ 0]$. We wish to construct a parameterization for $u_t$ that stabilizes $y_t$ near to zero where the noise is sampled as $w_t^* \sim U(-1,1)$. For the case where $T=1$, we list the values of $G_j^{(0)}$ below:
\begin{align*}
        G_1^{(0)} &=[0, \omega x_0, 0]^\top  &G_2^{(0)} &=[0, \mu y_0, 0]^\top \\
        G_3^{(0)} &=[0, a y_0 z_0, 0]^\top  &G_4^{(0)} &=[\mu x_0, 0, \lambda x_0^2]^\top \\
        G_5^{(0)} &=[-\omega y_0, 0, \lambda y_0^2]^\top &G_5^{(0)} &=[a x_0 y_0, 0, -\lambda z_0]^\top.
\end{align*}
For the sake of a readable example, we select $G_j^{(1)}$ such that the remaining terms at the next time are a function of only $w_t$, not $w_{t:t-1}$. The 17 terms satisfy
\begin{align*}
    &\omega w_t^x + \mu w_t^y + a w_t^y w_t^z + u_t \\ 
    &\qquad \quad = -\sum_{j=1}^{17} G_j^{(1)}(w_{t:t-1}) + \sum_{j=1}^6  (1-A_j^{(0)}) G_j^{(0)} (w_t).
\end{align*}
As in the first example, we fix all but one of the coefficients and sweep over a range of values for $A_2^{(2)}$. Because we have a larger space of disturbances, we run 1000 trials with disturbances generated according to $U(-1,1)$. The length of each trial is 23 time steps.
The results of our simulations are show in figure \ref{fig:fluid_results}. Again, we see that the cost is convex in $\alpha_2^{(0)}$ and smooth. We observe a tradeoff in the cost of the state versus the cost in the input; this corresponds to the disturbances decaying through the dynamics, which is controlled by the values of $\alpha$. 

\begin{figure}[t!]
  \centering
  \includegraphics[width=.235\textwidth,trim=0 0 0 0, clip]{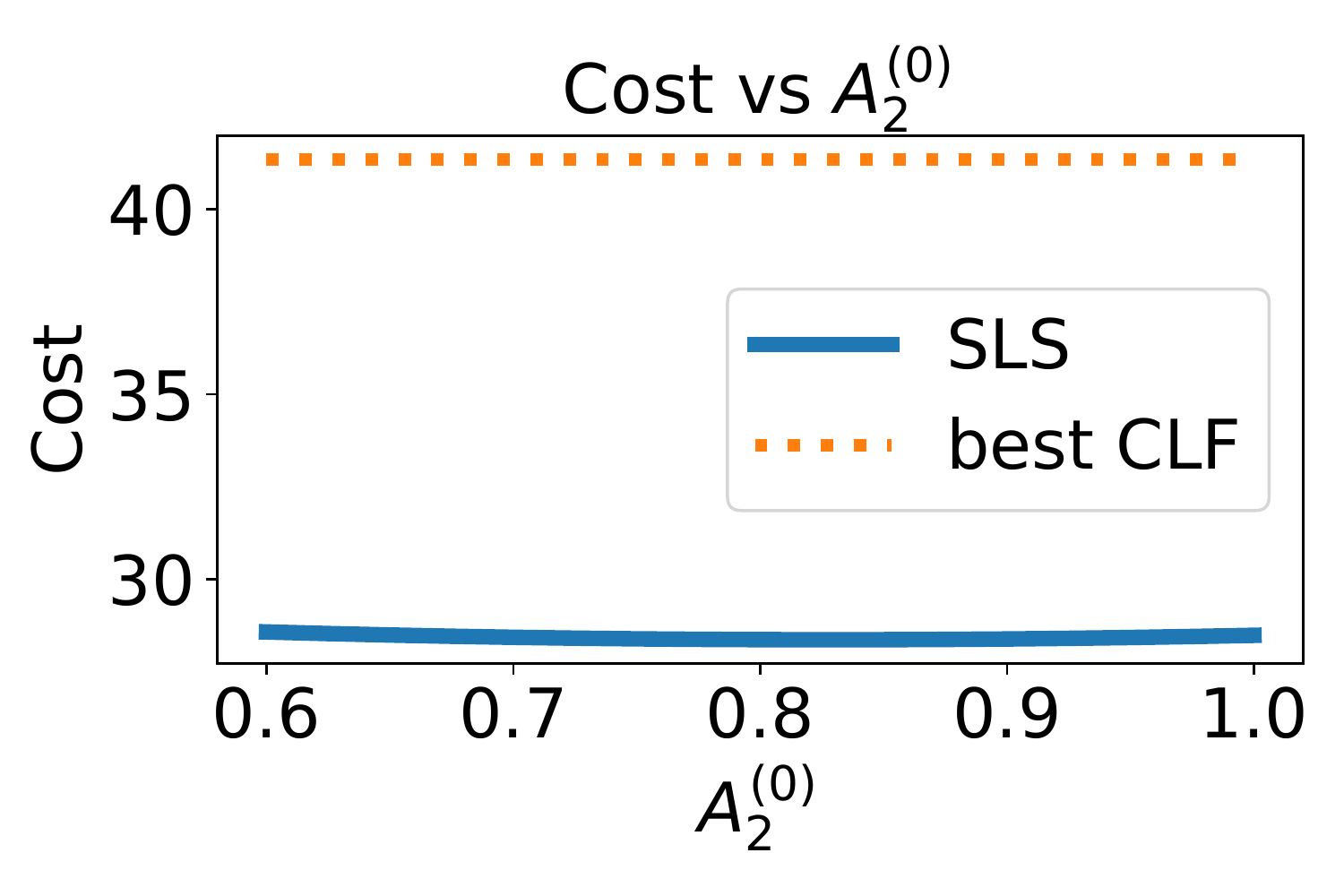}
  \includegraphics[width=.235\textwidth,trim=0 0 0 0, clip]{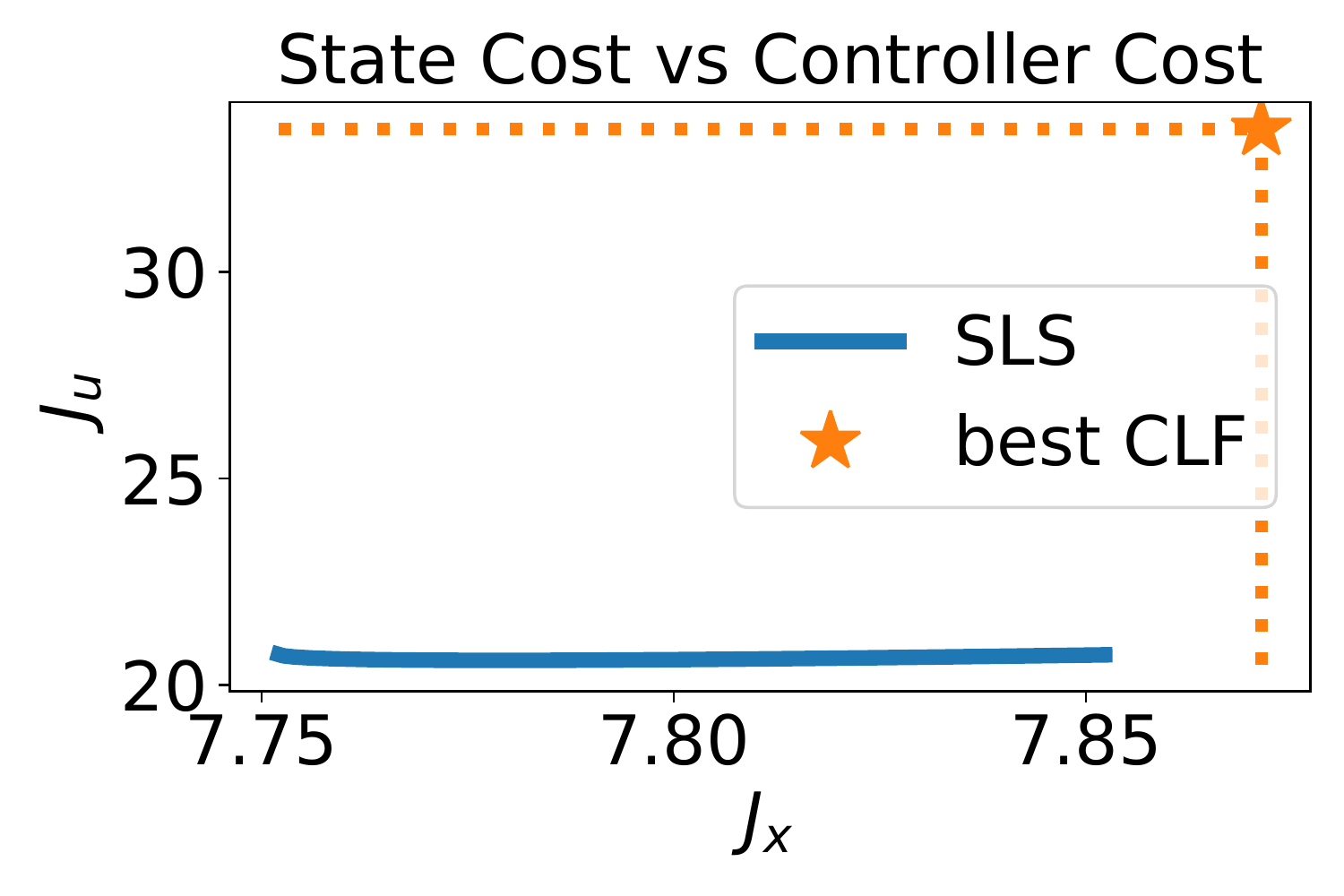}
  \includegraphics[width=.235\textwidth,trim=0 0 0 0, clip]{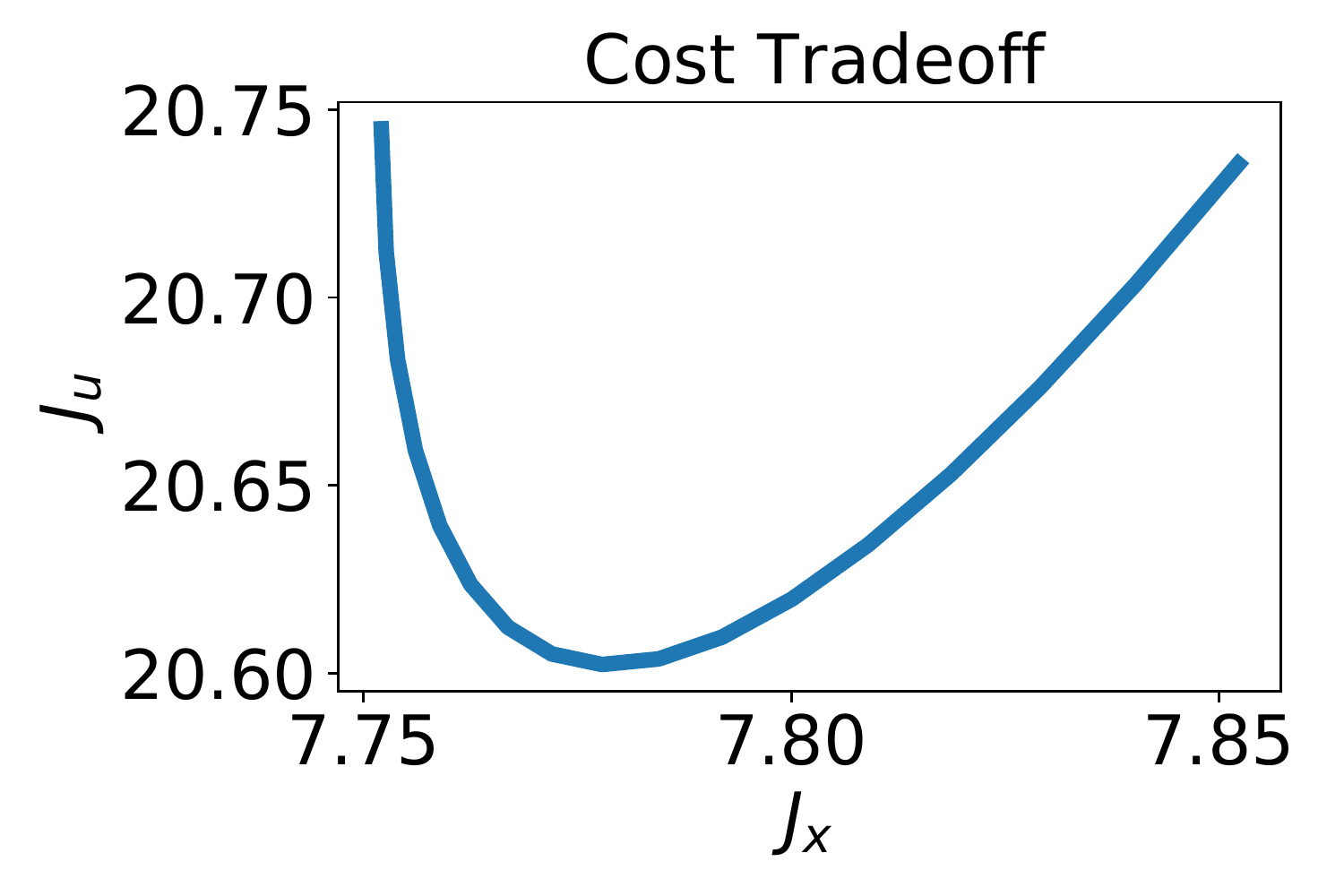}
  \includegraphics[width=.235\textwidth,trim=0 0 0 0, clip]{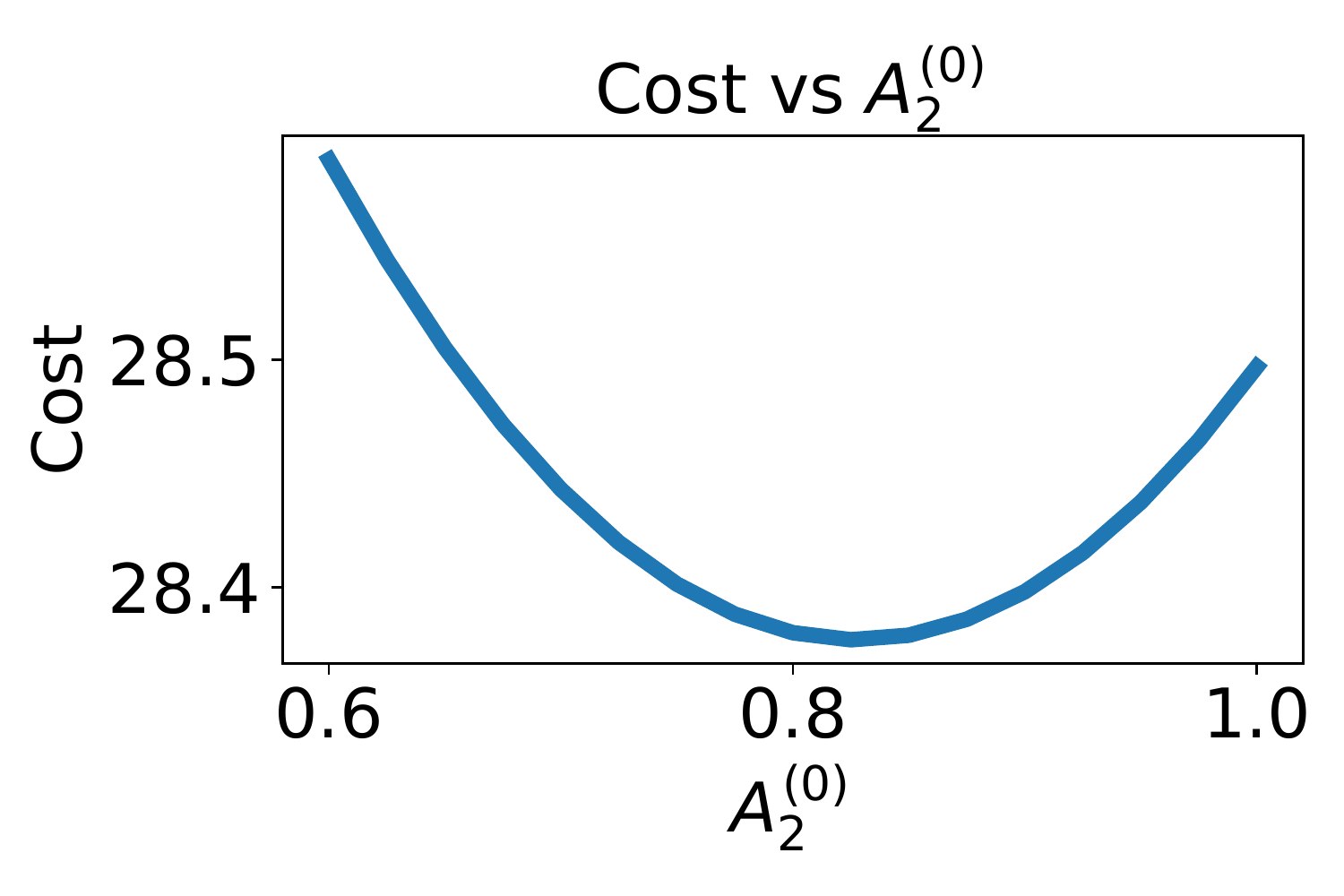}
  \vspace{-.15in}
  \caption{The value chosen for the $A_2^{(0)}$ for the vortex shedding model can be optimized with respect to the total cost.}
  \label{fig:fluid_results}
\end{figure}

\section{Conclusion}
In this work, we construct closed-loop, finite impulse response transfer functions for polynomial dynamical systems. The controller is parameterized by coefficients that determine the amount of feedback linearization for each monomial function of the disturbances. This offers a generalized perspective on feedback linearization, harnessing the framework of SLS for its construction. 

Because the controller is optimized offline, further research can investigate how to implement the coefficients as functions of the disturbances, offering additional flexibility and performance. These methods could also be extended to systems beyond the class of polynomials, or applied to other classes via polynomial approximation methods and lifting~\cite{Qian}. We are interested in the connection between our work and the dynamic programming literature, as we see parallels in iterating through a cost function over time. Because we recognize that the assumptions on \eqref{eq:scalar_polynomial} and \eqref{eq:vector_dynamics} are limiting, we wish to expand these classes to underactuated systems as well as develop appropriate definitions of controllability and stabilizability through this framework.
Other possible future work includes extending this framework to continuous-time systems, higher-order systems, and systems that are obtained through online learning. 

\bibliography{refs}
\bibliographystyle{IEEEtran}

\end{document}